\documentclass[11pt,reqno]{article}
\usepackage{latexsym, amsmath, amssymb, amsthm, a4, epsfig}
\usepackage{graphicx}
\usepackage{ascmac}
\usepackage{color}

\newtheorem{theorem}{Theorem}[section]

\newtheorem{lemma}[theorem]{Lemma}

\newcommand{\p}{\partial}

\newcommand{\eqnref}[1]{(\ref {#1})}

\newcommand{\Nbb}{\mathbb{N}}

\newcommand{\Rbb}{\mathbb{R}}

\newcommand{\la}{\langle}
\newcommand{\ra}{\rangle}

\newcommand{\Hcal}{\mathcal{H}}

\newcommand{\Kcal}{\mathcal{K}}

\newcommand{\Scal}{\mathcal{S}}

\newcommand{\Ga}{\alpha}
\newcommand{\Gb}{\beta}
\newcommand{\Gd}{\delta}
\newcommand{\Ge}{\epsilon}

\newcommand{\Gvf}{\varphi}

\newcommand{\Gk}{\kappa}

\newcommand{\Gl}{\lambda}

\newcommand{\Gs}{\sigma}

\newcommand{\GG}{\Gamma}

\newcommand{\GO}{\Omega}

\newcommand{\beq}{\begin{equation}}
\newcommand{\eeq}{\end{equation}}

\def\ol{\overline}

\numberwithin{equation}{section}
\numberwithin{figure}{section}

\begin{document}

\title{Surface localization of plasmons in three dimensions and convexity\thanks{\footnotesize This work was supported by NRF (of S. Korea) grants No. 2019R1A2B5B01069967 and by JSPS (of Japan) KAKENHI Grant Numbers JP19K14553, 19H01799.}}

\author{Kazunori Ando\thanks{Department of Electrical and Electronic Engineering and Computer Science, Ehime University, Ehime 790-8577, Japan. Email: {\tt ando@cs.ehime-u.ac.jp}.}
\and Hyeonbae Kang\thanks{Department of Mathematics and Institute of Applied Mathematics, Inha University, Incheon 22212, S. Korea. Email: {\tt hbkang@inha.ac.kr}.}
\and Yoshihisa Miyanishi\thanks{Center for Mathematical Modeling and Data Science, Osaka University, Osaka 560-8531, Japan. Email: {\tt miyanishi@sigmath.es.osaka-u.ac.jp, nakazawa@sigmath.es.osaka-u.ac.jp}.}
\and Takashi Nakazawa\footnotemark[4]}
\date{}
\maketitle

\begin{abstract}
The Neumann--Poincar\'e operator defined on a smooth surface has a sequence of eigenvalues converging to zero, and the single layer potentials of the corresponding eigenfunctions, called plasmons, decay to zero, i.e.,  are localized on the surface, as the index of the sequence $j$ tends to infinity. We investigate quantitatively the surface localization of the plasmons in three dimensions. The results are threefold. We first prove that on smooth bounded domains of general shape the sequence of plasmons converges to zero off the boundary surface almost surely at the rate of $j^{-1/2}$. We then prove that if the domain is strictly convex, then the convergence rate becomes $j^{-\infty}$, namely, it is faster than $j^{-N}$ for any integer $N$. As a consequence, we prove that cloaking by anomalous localized resonance does not occur on three-dimensional strictly convex smooth domains. We then look into the surface localization of the plasmons on the Clifford torus by numerical computations. The Clifford torus is taken as an example of non-convex surfaces. The computational results show that the torus exhibits the spectral property completely different from strictly convex domains. In particular, they suggest that there is a subsequence of plasmons on the torus which has much slower decay than other entries of the sequence.
\end{abstract}




\section{Introduction}

The Neumann--Poincar\'e (abbreviated by NP afterwards) operator defined on the boundary surface of a bounded domain is a self-adjoint compact operator on a proper Hilbert space provided that the surface is smooth enough. Thus it has a sequence of eigenvalues which converges to zero. The single layer potentials of the corresponding eigenfunctions are called plasmons. In the case of spheres, the plasmons are given by spherical harmonics, and one can easily see that plasmons, which are defined in the whole space, decay exponentially fast outside the sphere, the boundary of the domain. For this reason, the plasmon is usually said to be localized on the surface.

It is the purpose of this paper to investigate in a quantitative manner the surface localization of the plasmons on domains of general shape. We are particularly interested in how geometry, specifically the convexity, determines the localization property of the plasmons. It is helpful to mention here that the NP operator in three dimensions exhibits quite different spectral properties on convex and non-convex domains. For example, NP operators on convex domains have at most finitely many negative eigenvalues, while they have infinitely many negative eigenvalues on non-convex domains (see recent papers \cite{AJKKM, JK, MR}).

We now present the results of the paper in a precise manner. Let $\GO$ be a bounded domain in $\Rbb^3$. The NP operator on $\p\GO$ is well-defined if $\p\GO$ is Lipschitz continuous. However, for the purpose of this paper's investigation, we assume that $\p\GO$ is $C^{1,\Ga}$ for some $\Ga>0$. Let
\beq
\GG(x-y)=\frac{1}{4\pi|x-y|},
\eeq
the fundamental solution for the Laplacian in three dimensions. Then single layer potential operator $\Scal_{\p\GO}$ is defined as
\beq\label{Single layer}
\Scal_{\p\GO}[\psi](x):=\int_{\p\GO}\GG(x-y) \psi(y)\; dS_y, \quad x \in \Rbb^3,
\eeq
where $dS$ denotes the surface element on $\p\GO$. The potential function $\psi$ belongs to $H^{-1/2}(\p\GO)$, the $L^2$-Sobolev space of order $-1/2$ on the boundary surface $\p\GO$. We may regard $\Scal_{\p\GO}$ as a bounded operator from $H^{-1/2}(\p\GO)$ into $H^1(\Rbb^3)$, or from $H^{-1/2}(\p\GO)$ into $H^{1/2}(\p\GO)$. It is known that $\Scal_{\p\GO}: H^{-1/2}(\p\GO) \to H^{1/2}(\p\GO)$ is invertible \cite{Verch-JFA-84}. Furthermore, if we introduce a bilinear form on $H^{-1/2}(\p\GO)$ by
\beq\label{inner}
\la \psi, \Gvf \ra_{\Hcal^*} := \la \psi, \Scal_{\p\GO}[\Gvf] \ra,
\eeq
where $\la \cdot, \cdot \ra$ denotes the $H^{-1/2}-H^{1/2}$ pairing, then $\la \cdot, \cdot \ra_{\Hcal^*}$ is actually an inner product on $H^{-1/2}(\p\GO)$ and induces the norm equivalent to the usual Sobolev norm \cite{KKLSY-JLMS-16, KPS-ARMA-07}. Throughout this paper $\Hcal^*$ denotes the space $H^{-1/2}(\p\GO)$ equipped with the inner product \eqnref{inner}, and the norm is denoted by $\| \cdot \|_{\Hcal^*}$.

We now introduce the NP operator. Let for $\Gvf \in H^{1/2}(\p\GO)$
\beq\label{Double layer}
\Kcal_{\p\GO}[\Gvf](x):=\int_{\p\GO} \p_{\nu_y} \GG(x-y) \Gvf(y)\; dS_y.
\eeq
Here $\p_{\nu_y}$ denotes the outward normal derivative with respect to $y$-variable.  The operator $\Kcal_{\p\GO}$ is related to the single layer potential through the following jump relation:
\beq\label{singlejump}
\p_\nu \Scal_{\p \GO} [\Gvf] \Big|_\pm (x) = \biggl( \pm \frac{1}{2} I + \Kcal_{\p \GO}^* \biggr) [\Gvf] (x),
\quad x \in \p \GO\;,
\eeq
where the subscripts $\pm$ indicate the limit from outside and inside $\GO$, respectively, and the operator $\Kcal_{\p \GO}^*$ is the adjoint of $\Kcal_{\p\GO}$ on $L^2(\p\GO)$ (see, for example, \cite{AmKa07Book2, Fo95}). The operator $\Kcal_{\p\GO}$ (or $\Kcal_{\p \GO}^*$) is called the NP operator on $\p\GO$.

It is known that $\Kcal_{\p \GO}^*$ is self-adjoint with respect to the inner product \eqnref{inner}. In fact, it is a consequence of Plemelj's symmetrization principle (also known as Calder\'on's identity)
\beq\label{Plemelj}
\Scal_{\p \GO} \Kcal^*_{\p \GO} = \Kcal_{\p \GO} \Scal_{\p \GO}.
\eeq
Moreover, if $\p\GO$ is $C^{1,\Ga}$ (as we assume in this paper), $\Kcal_{\p \GO}^*$ is compact on $\Hcal^*$. So it has eigenvalues of finite multiplicities accumulating to $0$. We enumerate them as $\Gl_j$ counting multiplicities in the descending order in absolute value. Here and throughout this paper, we let $\Gvf_j$ ($j=0,1,2, \ldots$) be the corresponding eigenfunctions of $\Kcal_{\p \GO}^*$ normalized so that $\| \Gvf_j \|_{\Hcal^*}=1$ ($\Gvf_0$ is the eigenfunction corresponding to the simple eigenvalue $1/2$). The plasmons are therefore $\Scal_{\p\GO}[\Gvf_j]$, the single layer potentials of eigenfunctions. We emphasize that the plasmons defined in the whole space $\Rbb^3$.

If $\GO$ is the ball of radius $r_0$ (centered at $0$), then the NP eigenvalues are $\frac{1}{2(2n+1)}$ and their multiplicities are $2n+1$. Moreover, the corresponding eigenfunctions are $c_{n,m} Y_n^m$ where $Y_n^m$ are spherical harmonics and $c_{n,m}$ are constants chosen for normalization. So, in this case the plasmon is given by
$$
\Scal_{\p \GO}[c_{n,m} Y_n^m](x) =
\begin{cases}
d_{n,m} (r_0/r)^{n+1} Y_n^m(\hat{x}) \quad&\mbox{if } r=|x| \ge r_0,  \\
d_{n,m} (r/r_0)^{n+1} Y_n^m(\hat{x}) \quad&\mbox{if } r=|x| \le r_0,
\end{cases}
$$
where $\hat{x}=x/|x|$ and $d_{n,m}$ are constants which we do not specify. One can see from this explicit formula that in the case of balls the plasmons are localized on the boundary. In fact, one can see that plasmons decay exponentially fast as $n \to \infty$ if $x$ is not on $\p\GO$.
In this paper we quantitatively investigate the surface localization or equivalently the decay off the boundary surface of the plasmon when the surface $\p\GO$ is of arbitrary shape.

To present the first main result of this paper, some preparations are necessary. Given $\Ge >0$, let $\GO_\Ge$ stand for the $\Ge$-tubular neighborhood of $\p\GO$, namely,
\beq
\GO_\Ge :=\{\ x\in \Rbb^3 \ | \ \mbox{dist} (x, \p\GO) < \Ge \ \}.
\eeq
The support of a sequence of functions $\{u_j\}$ in $H^{1}(\Rbb^3)$ is defined as
$$
({\rm supp} \{u_j\})^c:= \{ x \in \Rbb^3\ |\ \mbox{$\exists$ a neighborfood $B_x$ of}\ x\ \mbox{\rm such that}\ \lim_{j\rightarrow 0} \Vert u_j \Vert_{H^{1}(B_x)}=0 \ \}.
$$
For a sequence $\{ a_j\}$ of numbers and a non-negative number $s$, we say $a_j=o(j^{-s})$ as $j\rightarrow \infty$ {\it almost surely} if
\beq\label{as}
\lim_{\Gd \downarrow 0}\limsup_{N\rightarrow \infty} \frac{\sharp \{ j<N\ : \ |a_j|>\Gd j^{-s} \}}{N}=0,
\eeq
or equivalently, if there is a subsequence $\{j_k\}$ such that $a_{j_k}=o(j_k^{-s})$ as $k \to \infty$ and
$\lim_{k \to \infty} {j_k}/{k} =1$.

The first result of this paper is the following theorem for domains of arbitrary shape.

\begin{theorem}\label{thm:general}
Suppose that $\GO$ is a bounded domain in $\Rbb^3$ with the $C^{1,\Ga}$-smooth boundary for some $\Ga >0$. Then,
\beq\label{supp}
\emptyset \neq \mbox{\rm supp} \{\Scal_{\p\GO}[\Gvf_j]\} \subset \p\GO.
\eeq
Furthermore, for any $\Ge >0$,
\beq\label{decayrate1}
\Scal_{\p\GO}[\Gvf_j]=o(1) \quad \mbox{in }\ C^{\infty}_{loc}(\GO_{\Ge}^c),
\eeq
and
\beq\label{decayrate2}
\Scal_{\p\GO}[\Gvf_j]=o(j^{-1/2}) \quad \mbox{almost surely in } \ C^{\infty}_{loc}(\GO_{\Ge}^c),
\eeq
as $j \to \infty$.
\end{theorem}

The statement \eqnref{decayrate1} means that for any compact subset $K$ of $\GO_\Ge^c$ and for any multi-index $\Gb$
\beq
\sup_{x\in K} |\p^\Gb \Scal_{\p\GO}[\Gvf_j](x)|=o(1) \quad\mbox{as } j \to \infty,
\eeq
and \eqnref{decayrate2} means likewise.

Theorem \ref{thm:general} is a geometry-independent result which holds for general domains. The following theorem shows that the decay rate of the surface localized plasmons is dramatically enhanced on strictly convex domains.

\begin{theorem}\label{thm:convex}
Let $\GO$ be a strictly convex bounded domain in $\Rbb^3$ with the $C^\infty$-smooth boundary. As $j \to \infty$,
\beq\label{decay-convex}
\Scal_{\p\GO}[\Gvf_j] =o(j^{-\infty}) \quad \mbox{\rm in}\ C_{loc}^{\infty}(\GO_{\Ge}^c),
\eeq
namely, $\Scal_{\p\GO}[\Gvf_j] =o(j^{-s})$ in $C_{loc}^{\infty}(\GO_{\Ge}^c)$ for all $s>0$.
\end{theorem}

This result has an important implication on plasmonic resonance. Such a fast decay of the plasmon off the boundary surface and the convergence of eigenvalues at the rate of $j^{-1/2}$, which was proved in \cite{Miyanishi:Weyl}, imply that cloaking by anomalous localized resonance (abbreviated by CALR) does not occur on strictly convex $C^\infty$-smooth bounded domains in $\Rbb^3$. A more precise statement of the result together with a proof and some historical accounts is provided in subsection \ref{subsec:CALR}.

The characteristic feature of the strictly convex domains in proving Theorem \ref{thm:convex} is that the NP operators have at most finitely many negative eigenvalues. But the NP operators on non-convex domains have infinitely many negative eigenvalues. These facts were proved in \cite{MR} (see also \cite{AJKKM} for the case of the Clifford torus). Thus, the method of the proof for Theorem \ref{thm:convex} does not apply to non-convex domains, and it is not known how fast the plasmon decays off the boundary surface except the slow convergence of the rate $j^{-1/2}$ given in Theorem \ref{thm:general}. Thus it is quite interesting to investigate if the plasmons on non-convex domains decay as fast as those on the convex domains. It is particularly interesting in relation to possibility of the cloaking by the anomalous localized resonance in three dimensions. Because of such interest and lack of tools for investigation, we investigate surface localization of plasmon on non-convex domains by numerical computations. We take the Clifford torus as a typical example of a non-convex domain, and compute plasmon $\Scal_{\p\GO}[\Gvf_j]$ for $j$ up to 450. The computational result shows, to our surprise, that there are four exceptional eigenvalues among 450 eigenvalues at which the plasmon has much lager values compared to other values. It suggests that there might be a subsequence $\Scal_{\p\GO}[\Gvf_{j_k}]$ which has slow decay. One intriguing feature of eigenfunctions corresponding to the four exceptional eigenvalues is that they are invariant under rotation with respect to the axis of symmetry. We also compute plasmons on an ellipsoid for the purpose of comparing with the case of the torus. The computational result exhibits clear contrast: there are no exceptional eigenvalues on the ellipsoid as Theorem \ref{thm:convex} suggests.

This paper is organized as follows. In section \ref{sec: Proof thm general}, we prove Theorem \ref{thm:general}.
Section \ref{sec: Proof thm convexl} is to prove Theorem \ref{thm:convex} and non-occurrence of CALR on a strictly convex $C^\infty$-smooth bounded domain. The results of numerical computations are presented
in section \ref{sec: Numerical non-convex}. This paper ends with a short discussion.

\section{General domains-Proof of Theorem \ref{thm:general}}\label{sec: Proof thm general}

In this section we prove Theorem \ref{thm:general}. To do so, we first prove the following lemma. We mention that arguments for proving the lemma are well known, for example, in quantum ergodicity for the Laplace eigenfunctions \cite{Zelditch}.

\begin{lemma}\label{almost surely}
If a numerical sequence $\{a_j\}$ satisfies $\sum_{j=1}^{\infty} |a_j|^2 <+\infty$, then
$a_j=o(j^{-1/2})$ as $j\rightarrow \infty$ almost surely.
\end{lemma}
\begin{proof}
Assume it is not true that $a_j=o(j^{-1/2})$ as $j\rightarrow \infty$ almost surely. In view of the definition \eqnref{as} of the almost sure convergence, we see that
$$
\lim_{\Gd\downarrow 0}\limsup_{N\rightarrow \infty} \frac{\sharp \{ j<N\ : \ |a_j|>\Gd j^{-1/2}\}}{N}> \Gd'
$$
for some $\Gd'>0$. Then there exists $\Gd>0$ such that
\beq\label{limsup}
\limsup_{N\rightarrow \infty} \frac{\sharp \{ j<N\ : \ |a_j|>\Gd j^{-1/2}\}}{N}>\Gd'.
\eeq
Thus there is an integer $N_1$ such that
$$
\sharp \{ j<N_1 \ : \ |a_j|>\Gd j^{-1/2}\}>\frac{1}{2} \Gd' N_1.
$$
Let $\{  j_1, j_2, \ldots, j_{k_1} \}$ be the enumeration of the set $\{ j<N_1 \ : \ |a_j|>\Gd j^{-1/2}\}$ in the ascending order. Then $k_1 > \Gd' N_1/2$ and
\begin{align*}
S_{N_1} &:=\sum_{j=1}^{N_1} |a_j|^2 \geq \sum_{k=1}^{k_1} \Gd^2 j_k^{-1} \geq \sum_{j=[(1-\frac{\Gd'}{2})N_1]-1}^{N_1} \Gd^2 j^{-1} \\
&\geq C\Gd^2 \left\{\log(N_1+1)-\log \left((1-\frac{\Gd'}{2}) N_1\right)\right\} \ge C' \Gd^2 |\log(1-\frac{\Gd'}{2})|
\end{align*}
for some constant $C$ and $C'$.

Thanks to \eqnref{limsup}, we see that there is an integer $N_2>N_1$ such that
$$
\sharp \{ N_1 \le j<N_2 \ : \ |a_j|>\Gd j^{-1/2}\}>\frac{1}{2} \Gd' N_2.
$$
Thus, through the same reasoning, we arrive at
\begin{align*}
S_{N_1, N_2}&:=\sum_{j=N_1+1}^{N_2} |a_j|^2 \geq \sum_{j=[(1-\frac{\Gd'}{2})N_2]-1}^{N_2} \Gd^2 j^{-1} \\
&\geq C\Gd^2 \left\{\log(N_2+1)-\log \left((1-\frac{\Gd'}{2}) N_2 \right)\right\} \ge C' \Gd^2 |\log(1-\frac{\Gd'}{2})|.
\end{align*}

One can inductively find the sequence $N_1< N_2< N_3 < \ldots$ such that
\begin{align*}
S_{N_{k}, N_{k+1}}&:=\sum_{j=N_k+1}^{N_{k+1}} |a_j|^2 \geq C' \Gd^2 |\log(1-\frac{\Gd'}{2})|
\end{align*}
for all $k$. We thus have
$$
\sum_{j=1}^{\infty} |a_j|^2 >S_{N_1}+S_{N_1, N_2}+S_{N_2, N_3}+ \cdots = \infty,
$$
which is the desired contradiction. Thus $a_j=o(j^{-1/2})$ as $j\rightarrow \infty$ almost surely.
\end{proof}

\noindent{\sl Proof of Theorem \ref{thm:general}}.
We first note that if we define
\beq\label{inner2}
\la f, g \ra_{\Hcal} := \la \Scal_{\p\GO}^{-1}[f], \Scal_{\p\GO}^{-1}[g] \ra_{\Hcal^*}
\eeq
for $f,g \in H^{1/2}(\p\GO)$, then $\la \cdot, \cdot \ra_{\Hcal}$ is an inner product on $H^{1/2}(\p\GO)$. Let $\Hcal=\Hcal(\p\GO)$ be the space $H^{1/2}(\p\GO)$ equipped with the inner product $\la \cdot, \cdot \ra_{\Hcal}$, which induces a norm equivalent to $H^{1/2}(\p\GO)$-norm. Then, eigenvalues of $\Kcal_{\p \GO}^*$ on $\Hcal^*$ and $\Kcal_{\p \GO}$ on $\Hcal$ are identical, and functions $\Scal_{\p \GO}[\Gvf_j]$ on $\p\GO$ are normalized eigenfunctions of $\Kcal_{\p \GO}$ on $\Hcal$ ($\Gvf_j$ are normalized eigenfunctions of $\Kcal_{\p \GO}^*$ on $\Hcal^*$).

Let $u_j(x) = \Scal_{\p\GO}[\Gvf_j](x)$ for $x \in \Rbb^3$ and $j=0,1, \ldots$. It is proved in \cite[Theorem 3.1]{AK} that $\GG(x-z)$, the fundamental solution of Laplacian, admits the following expansion:
\beq\label{single layer sum of fundamental sol}
\GG(x-z)=-\sum_{j=1}^{\infty} u_j(z) u_j(x)+ u_0(z),\quad x \in \p \GO,\ z\in {\Rbb^d} \setminus \p\GO.
\eeq
Since $\{ u_j|_{\p\GO} \}$  is orthonormal in $\Hcal=\Hcal(\p\GO)$, we have
for fixed $z\in {\Rbb}^3\setminus \p\GO$,
\beq\label{sum is finite}
\sum_{j=1}^{\infty} |u_j(z)|^2 = \Big\Vert \sum_{j=1}^{\infty} u_j(z) u_j \Big \Vert^2_{\Hcal} \le \| \GG(z-\cdot) \|^2_{\Hcal}.  
\eeq
Since the $\Hcal$-norm is equivalent to the $H^{1/2}$-norm on $\p\GO$, we have for any compact set $K$ in $\Rbb^3 \setminus \p\GO$
\beq\label{est1}
\sum_{j=1}^{\infty} \int_K |u_j(z)|^2 dz \le \int_K \| \GG(z-\cdot) \|^2_{\Hcal} dz   \le C
\eeq
for some constant $C$. Thus, as $j \to \infty$,
\beq\label{est2}
\int_K |u_j(z)|^2 dz \to 0
\eeq
and
\beq\label{est3}
\int_K |u_j(z)|^2 dz =o(j^{-1}) \quad \mbox{almost surely}
\eeq
by Lemma \ref{almost surely}.

Fix a compact set $K$ in $\Rbb^3 \setminus \p\GO$ and let $K'=\ol{K_\Ge}= \{ x \in \Rbb^3 : \mbox{dist}(x, K) \le \Ge \}$. If $\Ge$ is small enough, then $K'$ is also compact in $\Rbb^3 \setminus \p\GO$. Since $u_j$ is harmonic in $\Rbb^3 \setminus \p\GO$, for any multi-index $\Gb$ there is a constant $C=C(\Gb, \Ge)$ such that if $x \in K$, then
$$
|\p^\Gb u_j(x)|^2 \le C \int_{B_\Ge(x)} |u_j(z)|^2 dz,
$$
where $B_\Ge(x)$ denotes the ball of radius $\Ge$ centered at $x$. Thus we have
\beq\label{est4}
\sup_{x \in K} |\p^\Gb u_j(x)|^2 \le C \int_{K'} |u_j(z)|^2 dz.
\eeq
Now, the fact that $\mbox{\rm supp} \{u_j\} \subset \p\GO$ (the second statement of \eqnref{supp}), \eqnref{decayrate1} and \eqnref{decayrate2} follow from \eqnref{est2}, \eqnref{est3} and \eqnref{est4}.

We now prove the first statment of \eqnref{supp}, namely, $\mbox{\rm supp} \{u_j\} \neq \emptyset$.
Since $\{ u_j \}$  is orthonormal in $\Hcal=\Hcal(\p\GO)$, there is a constant $C>0$ such that
$$
\Vert u_j \Vert_{H^{1/2}(\p\GO)} \ge C
$$
for all $j$. Thus there exists a point $x\in \p\GO$ such that for any ball $B_x$ centered at $x$
$$
\limsup_{j\rightarrow \infty}\Vert u_j \Vert_{H^{1/2}(B_x\cap \p\GO)} >0.
$$
It then follows from extension theorems of Sobolev spaces that
$$
\limsup_{j\rightarrow \infty}\Vert u_j \Vert_{H^1(B_x)} \ge C \limsup_{j\rightarrow \infty}\Vert u_j \Vert_{H^{1/2}(B_x \cap \p\GO)} >0.
$$
This means $x\in \mbox{\rm supp} \{u_j\} $ as desired. This completes the proof. \qed

\section{Strictly convex domains}\label{sec: Proof thm convexl}

In this section we prove Theorem \ref{thm:convex}. We then show as a consequence that CALR does not occur on the strictly convex domains in three dimensions.

\subsection{Proof of Theorem \ref{thm:convex}}

Since $\p\GO$ is $C^\infty$-smooth by assumption, the operator $\Kcal_{\p\GO}$ is known to be a strictly homogeneous pseudo-differential operator ($\Psi DO$) of order $-1$ (see \cite{Agrano-book}), and it is proved in \cite{Miyanishi:Weyl, MR} that its principal symbol is given by
\beq\label{NPsymbol}
\Gs_{prin}(\Kcal_{\p\GO}) \equiv \frac{L(x) \xi_2^2 -2M(x)\xi_1 \xi_2 +N(x)\xi_1^2}{4 \det (g_{ij})  \big( \sqrt{\sum_{j, k} g^{jk}(x)\xi_j \xi_k} \big)^{3}},
\eeq
where $g_{ij}$ is the metric tensor and $L(x), M(x), N(x)$ are the coefficients of the second fundamental form at the point $x$.
Since the operator $\Kcal^*_{\p\GO}-\Kcal_{\p\GO}$ is a $\Psi DO$ of order $-2$, $\Kcal^*_{\p\GO}$ and $\Kcal_{\p\GO}$ have the same principal symbols.
Thus the symbol $\Gs_{prin}(\Kcal_{\p\GO}^*)$ is positive definite if $\GO$ is strictly convex. In fact, it is a consequence of \eqnref{NPsymbol} and the relation
$$
\Gk=\frac{M^2-LM}{g_{11} g_{22} - g_{12}^2},
$$
where $\Gk$ is the Gaussian curvature. Thus all eigenvalues, except possibly finitely many, of $\Kcal_{\p\GO}^*$ are positive (see, e.g., \cite{Agrano-book, MR}). If $\Kcal_{\p\GO}^*$ has actually a finitely many non-positive eigenvalues, then we can modify it by adding a finite dimensional smoothing operator to have a positive definite $\Psi DO$ of order $-1$, which we denote by $\Kcal^*$. Thus for each real number $s$ there are constants $c_s$ and $C_s$ such that
\beq\label{Kcal}
c_s \Vert \Gvf \Vert_{H^{s-1/2}(\p\GO)} \leq \Vert \Kcal^* [\Gvf] \Vert_{H^{s+1/2}(\p\GO)} \leq C_s\Vert \Gvf \Vert_{H^{s-1/2}(\p\GO)}
\eeq
for all $\Gvf \in H^{s}(\p\GO)$. Note that there is $j_0$ such that
$$
\Kcal^* [\Gvf_j] = \Kcal_{\p\GO}^* [\Gvf_j] \quad \mbox{for all } j \ge j_0,
$$
and the corresponding eigenvalues $\Gl_j$ are all positive. We then infer from \eqnref{Kcal} that $\Gvf_j \in H^{s}(\p\GO)$ for all $s$ if $j \ge j_0$.

Let $K$ be a compact set in $\Rbb^3 \setminus \p\GO$. Since $\mbox{dist}(K, \p\GO) >0$, for any positive integer $k$ and for any real number $s$, there is a constant $M_{k,s}$ such that
$$
\Vert \Scal_{\p\GO}[\Gvf_j] \Vert_{C^k(K)} \leq M_{k,s} \Vert \Gvf_j \Vert_{H^{-s-1/2}(\p\GO)}.
$$
It then follows from \eqnref{Kcal} that
\beq\label{Gljs}
\Vert \Scal_{\p\GO}[\Gvf_j] \Vert_{C^k(K)} \leq C_{k,s} \Vert (\Kcal_{\p\GO}^*)^s [\Gvf_j] \Vert_{H^{-1/2}(\p\GO)} = C_{k,s} \Gl_j^{s}
\eeq
for some constant $C_{k,s}$ depending on $k$ and $s$, but independent of $j$.

It is proved in \cite{Miyanishi:Weyl} that $\{\Gl_j\}_{j\in \Nbb}$ asymptotically behaves like
\beq\label{weyl}
\Gl_j \sim C_{\p\GO} j^{-1/2}\quad \mbox{as } j\rightarrow \infty,
\eeq
in the sense that $\Gl_j j^{1/2} \to C_{\p\GO}$ as $j \to \infty$. Here the constant $C_{\p\GO}$ is given by
\beq\label{CpGO}
C_{\p\GO} = \Big( \frac{3W(\p\GO) - 2\pi \chi(\p\GO)}{128 \pi} \Big)^{1/2} ,
\eeq
where $W(\p\GO)$ and $\chi(\p\GO)$, respectively, denote the Willmore energy  and the Euler characteristic of the boundary surface $\p\GO$.
Thus we infer from \eqnref{Gljs} that for each integer $k$ and $s$ there is a constant $C_{k,s}$ such that
\beq
\Vert \Scal_{\p\GO}[\Gvf_j] \Vert_{C^k(K)} \leq C_{k,s} j^{-s/2} \quad\mbox{for all } j \ge j_0.
\eeq
So, we have \eqnref{decay-convex} and the proof is complete.
\qed

\subsection{Non-occurrence of CALR}\label{subsec:CALR}

Let $\GO$ be a bounded domain in $\Rbb^3$ made of a plasmonic material of a negative dielectric constant. The material property of the domain $\GO$ is represented by $\Ge_c+i\Gd$ where $\Ge_c<0$
indicates the negative dielectric constant and $\Gd>0$ does the dissipation. Let $\Ge_m>0$ be the dielectric constant of the
matrix $\Rbb^3\setminus \overline{\GO}$. Thus the distribution of the dielectric constant of the structure is given by
\beq
\Ge=
\begin{cases}
\Ge_c+i\Gd\quad &\mbox{in}\ \GO, \\
\Ge_{m}  \quad\quad\quad &\mbox{in}\ \Rbb^3\backslash \overline{\GO}.
\end{cases}
\eeq
We then consider the following problem:
\beq\label{harmonic eq}
\begin{cases}
\nabla\cdot \Ge \nabla u= a\cdot \nabla \Gd_z\quad \mbox{in}\ \Rbb^3, \\
u(x)=O(|x|^{-2}) \quad\ \ \mbox{as}\ |x|\rightarrow \infty,
\end{cases}
\eeq
where $a$ is a constant vector and $\Gd_z$ is the Dirac mass at $z\in \Rbb^3 \setminus \overline{\GO}$.

The CALR is a spectral phenomenon occurring at the accumulation point $0$ of the eigenvalues, for that we assume $\Ge_c+\Ge_m=0$. When the domain is an annulus in two dimensions, it is shown in \cite{MN_06} that there is a virtual radius such that the location $z$ of the dipole source appearing in \eqnref{harmonic eq} lies inside the virtual radius, then huge resonance occurs and the source becomes invisible. Here we only recall, referring to \cite{ACKLM-13, AK} for the definition of CALR,  that CALR is characterized by the quantity $E_\Gd(z):=\Gd \| \nabla u_\Gd \|_{L^2(\GO)}^2$ which represents the imaginary part of the energy, namely,
$$
E_\Gd(z) = \Im \int_{\Rbb^d} \Ge |\nabla u_\Gd|^2 dx.
$$
Physically it represents the electro-magnetic energy dissipated into heat. For CALR to occur, it is necessary for $E_\Gd(z)$ to tend to $\infty$ as $\Gd \to 0$ for some $z$. We prove the following theorem which implies that the CALR does not occur on strictly convex domains in $\Rbb^3$. CALR is known to occur on concentric disks, ellipses, and confocal ellipses, and not to occur on balls and concentric balls (see, for example, \cite{ACKLM-13, ACKLM14, AK, CKKL, KLSW, MN_06}). We also refer to a recent review of McPhadran and Milton \cite{MM} for CALR and related topics including some historic accounts.

\begin{theorem}
Let $\GO$ be a strictly convex bounded domain in $\Rbb^3$ with the $C^\infty$-smooth boundary and assume that $\ker (\Kcal_{\p\GO})=\emptyset$. The following holds for the solution $u_\Gd$ to \eqnref{harmonic eq}: For any $z$ there is $C=C_z$ independent of $\Gd$ such that
\beq\label{ALRbound}
\| \nabla u_\Gd \|_{L^2(\GO)} < C.
\eeq
In particular, $E_\Gd(z) \to 0$ as $\Gd \to 0$.
\end{theorem}

\begin{proof}
It is proved in \cite{AK} that
\beq\label{fundamental relation of resonance}
\Vert \nabla u_\Gd \Vert_{L^2(\GO)}^2 \sim \sum_{j=1}^{\infty} \frac{|a\cdot \nabla \Scal_{\p\GO}[\Gvf_j](z)|^2}{\Gd^2+\Gl_j^2}.
\eeq
According to Theorem \ref{thm:convex}, for any positive integer $N$ there is a constant $C_N$ such that
\beq\label{100}
|a\cdot \nabla \Scal_{\p\GO}[\Gvf_j](z)| \le C_N j^{-N}
\eeq
for all $j$ large enough. Since $W(\p\GO) \ge 4\pi$ (see, e.g., \cite{MN}) and $\chi(\p\GO)=2$ if $\p\GO$ is strictly convex, $C_{\p\GO} >0$. Thus \eqnref{100} together with \eqnref{weyl} leads us to
$$
\Vert \nabla u_\Gd \Vert_{L^2(\GO)}^2 \le C \sum_{j=1}^{\infty} \frac{j^{-2N}}{\Gd^2+j^{-1}},
$$
from which \eqnref{ALRbound} immediately follows.
\end{proof}

\section{Non-convex domains-Numerical experiments}\label{sec: Numerical non-convex}

The purpose of this section is to present results of numerical experiments for the surface localization of the plasmon on non-convex domains. Our interest lies in the question if the plasmon $\Scal_{\p\GO}[\Gvf_j](z)$ decays like $o(j^{-\infty})$ on non-convex surfaces like the strictly convex case. Throughout this section we use $u_j$ for $\Scal_{\p\GO}[\Gvf_j]$ for simplicity of notation.

We consider the Clifford torus as an example of non-convex surfaces:
$$
\p\GO:=T^2_{\text{Clifford}}:=\{ ((\sqrt{2}+\cos u)\cos v, (\sqrt{2}+\cos u)\sin v, \sin u) \in {\Rbb}^3\ |\ u, v\in [0, 2\pi)\ \}.
$$
We compute, for $j=1, 2, \ldots, 450$,
$$
\Vert u_j \Vert_{L^{2}(X)} = \left( \int_{X} |u_j(z)|^2 dz \right)^{1/2},
$$
where the region of integration $X$ is given by
\beq\label{setX}
X:= \{(x, 0, z) \in  \Rbb^3 \ |\ 0< x < 2\sqrt{2}, \ 0< z < 2\sqrt{2}, \ \mbox{dist} ((x, 0, z), \GO) >\Ge \}.
\eeq
Here $\Ge$ is the the side length of the typical mesh used for computation: $\Ge$ is approximately $1/\sqrt{2,000,000}$ as will be explained later. We emphasize that since $\Ge$ is quite small, the region $X$ of integration almost touching $\p\GO$ and this prevent $\Vert u_j \Vert_{L^{2}(X)}$ from being too small too quickly as $j$ is increasing. The result of computations is presented in Figure \ref{fig:log_energy}: it shows the graph of $5 \log \Vert u_j \Vert_{L^{2}(X)} + 35$. Here, the numbers $5$ and $35$ are chosen for nothing other than clarity of the presentation.

For the purpose of comparison, we present computational results of $\Vert u_j \Vert_{L^{2}(Y)}$ when the given domain is the oblate spheroid:
$$
\frac{x^2}{2}+\frac{y^2}{2}+z^2=1.
$$
Here the region of integration $Y$ is given by
\beq\label{setY}
Y := \{(x, 0, z) \in  \Rbb^3 \ |\  0 <x <3\sqrt{2}/2 , \ 0<z<2, \ \mbox{dist} ((x, 0, z), E)>\Ge \}
\eeq
with $E$ being the ellipsoid. It is known that the NP eigenvalues on oblate spheroids are positive except finitely many negative eigenvalues \cite{MR} (it is helpful to mention that some oblate spheroids actually have negative NP eigenvalues \cite{Ahner}). 
Figure \ref{fig:ellipsoid} present the graph of $\log \Vert u_j \Vert_{L^{2}(Y)}$.

\begin{figure}[t]
\centering
\includegraphics[width=.5\textwidth]{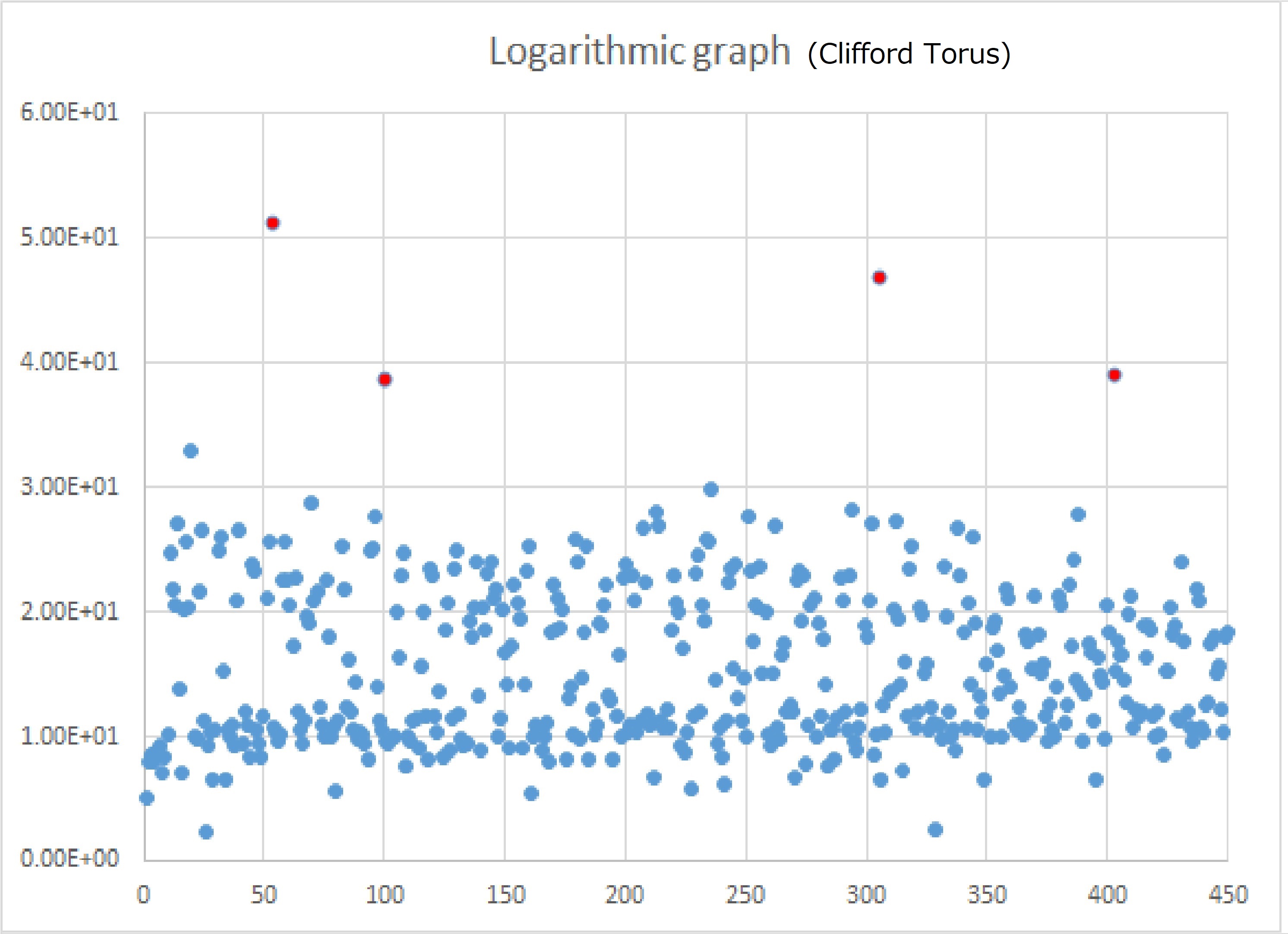}
\caption{The graph of $5 \log \Vert u_j \Vert_{L^{2}(X)} + 35$ on the Clifford torus. The horizontal axis represents positive eigenvalues of the NP operator enumerated in decreasing order up to 450. The red dots indicate values drastically larger than neighboring points. They occur at 53rd, 100th, 305th and 402nd eigenvalues. }
\label{fig:log_energy}
\end{figure}

\begin{figure}[htp]
\centering
\includegraphics[width=.5\textwidth]{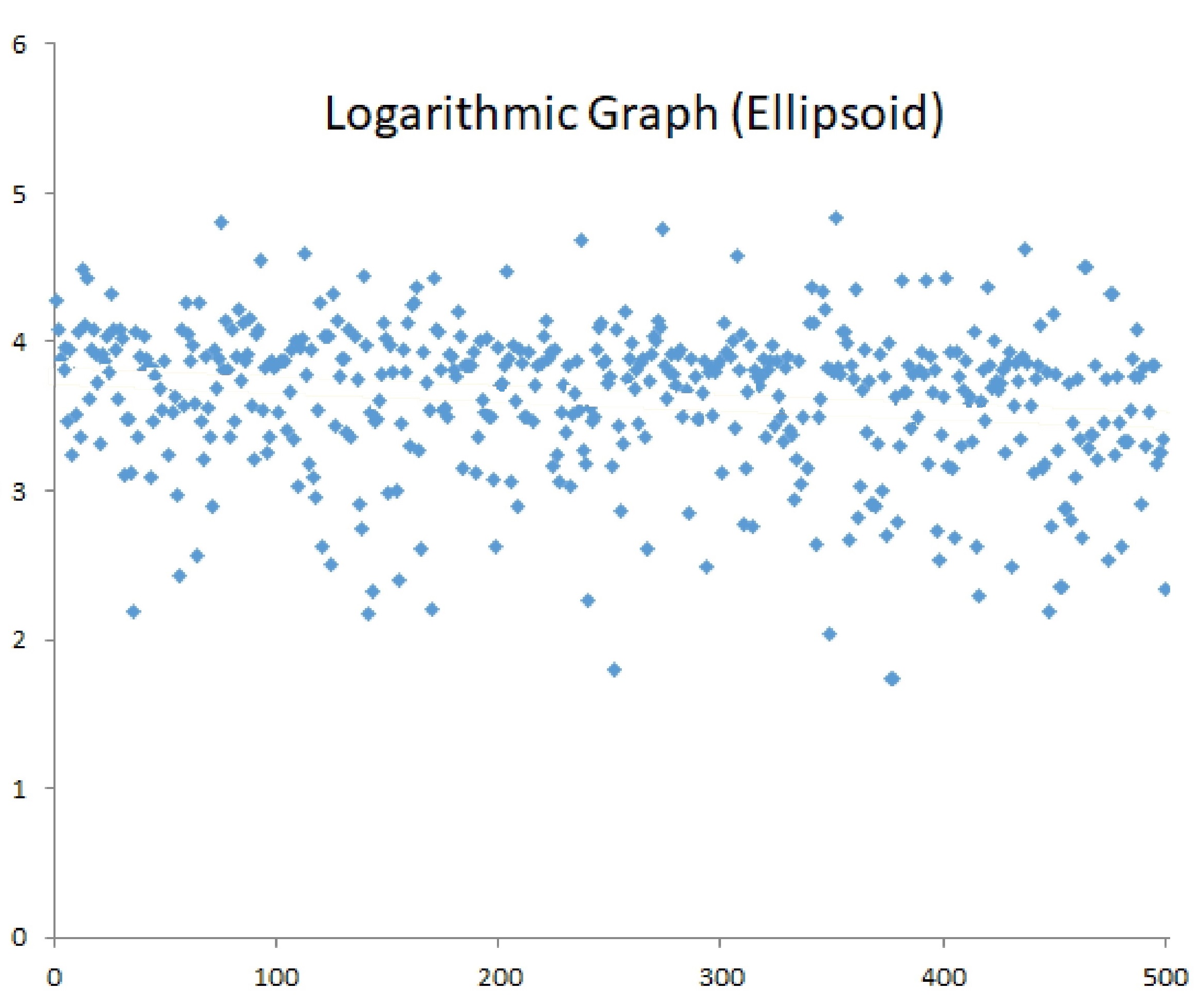}
\caption{The graph of $\log \Vert u_j \Vert_{L^{2}(Y)}$ on the ellipsoid. The horizontal axis represents positive eigenvalues of the NP operator enumerated in decreasing order up to 500. There is no exceptionally large value.}
\label{fig:ellipsoid}
\end{figure}

Figure \ref{fig:log_energy} shows that in the case of the torus there are four values which are drastically larger than the other ones, and they occur at 53rd, 100th, 305th and 402nd eigenvalues. The computation of NP eigenvalues on the surface in three dimensions is quite difficult, and the numerical errors prevent computations to be carried out beyond 450th eigenvalues. However, the computational result strongly suggests that in the case of the Clifford torus there might be a subsequence of plasmons which does not decay rapidly. On the other hand, according to Figure \ref{fig:ellipsoid}, no exceptionally large values appear in the case of the ellipsoid which is in accordance with the theoretical result (Theorem \ref{thm:convex}).

We look into the exceptional values further by investigating corresponding eigenfunctions. Figure \ref{fig:eigenfunction1} shows the eigenfunction corresponding 53rd eigenvalue (exceptional one) compared with that corresponding to the 52nd eigenvalue. The figure exhibits an intriguing feature of eigenfunctions corresponding to exceptional eigenvalues: The 53rd eigenfunction is invariant under the rotation with respect to the $z-$axis, but $52$nd one is not. Figure \ref{fig:eigenfunction2} shows that eigenfunctions corresponding to the other three exceptional eigenvalues have the same property. It also shows that the other (401st and 403rd) eigenfunctions are not invariant under the rotation with respect to the $z-$axis.

Such exceptional values in the torus also occur among negative eigenvalues. For example, the $-39$th eigenvalue is an exceptional one. Figure \ref{fig:eigenfunction3} shows the corresponding eigenfunction compared with the $-38$th: the $-39$th eigenfunction is invariant under the rotation with respect to the $z-$axis, but the $-38$th eigenfunction is not.

\begin{figure}[htp]
\centering
\includegraphics[width=.39\textwidth]{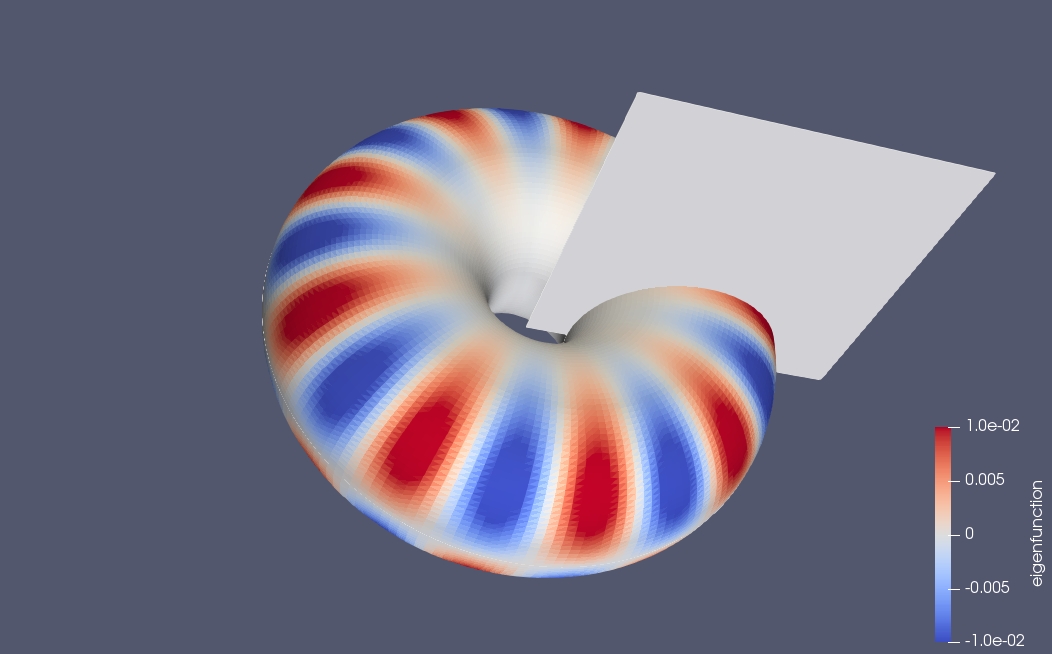}\hspace{.01\textwidth}
\includegraphics[width=.39\textwidth]{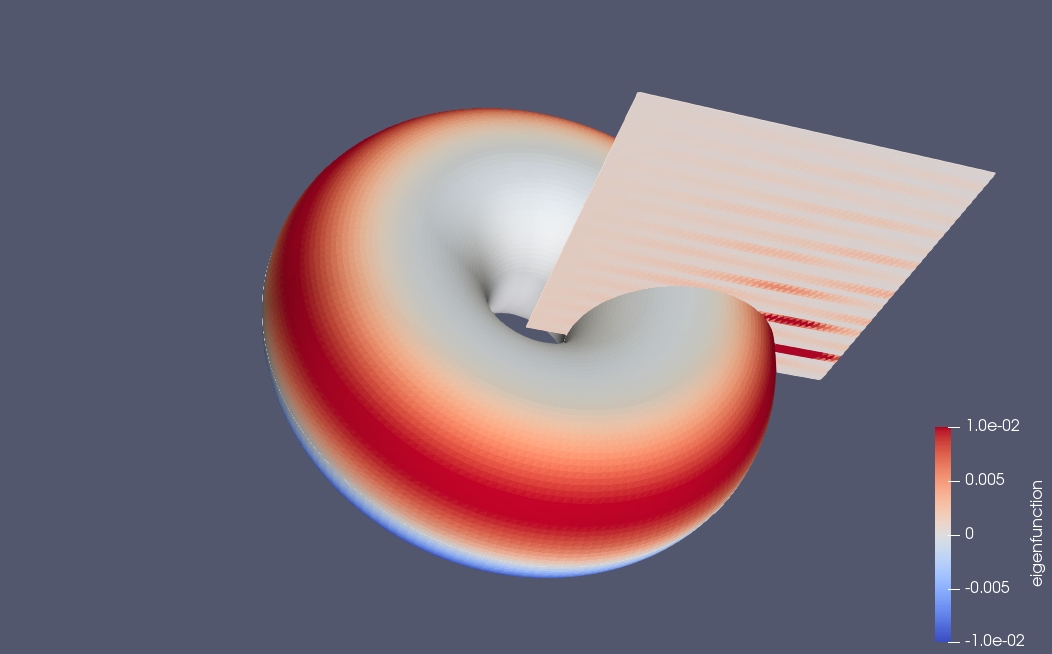}
\caption{The $52$nd eigenfunction (left) and the $53$th eigenfunction (right) on the Clifford torus. The $53$th eigenfunction is invariant under the rotation with respect to the $z-$axis, but $52$nd one is not. The $53$th eigenvalue is an exceptional one. The rectangular cross section represents the region of integration $X$ in \eqnref{setX}, and the color on it represents the value of $u_j(z)= \Scal_{\p\GO}[\Gvf_j](z)$ ($z \in X$).}
\label{fig:eigenfunction1}
\end{figure}

\begin{figure}[htp]
\centering
\includegraphics[width=.30\textwidth]{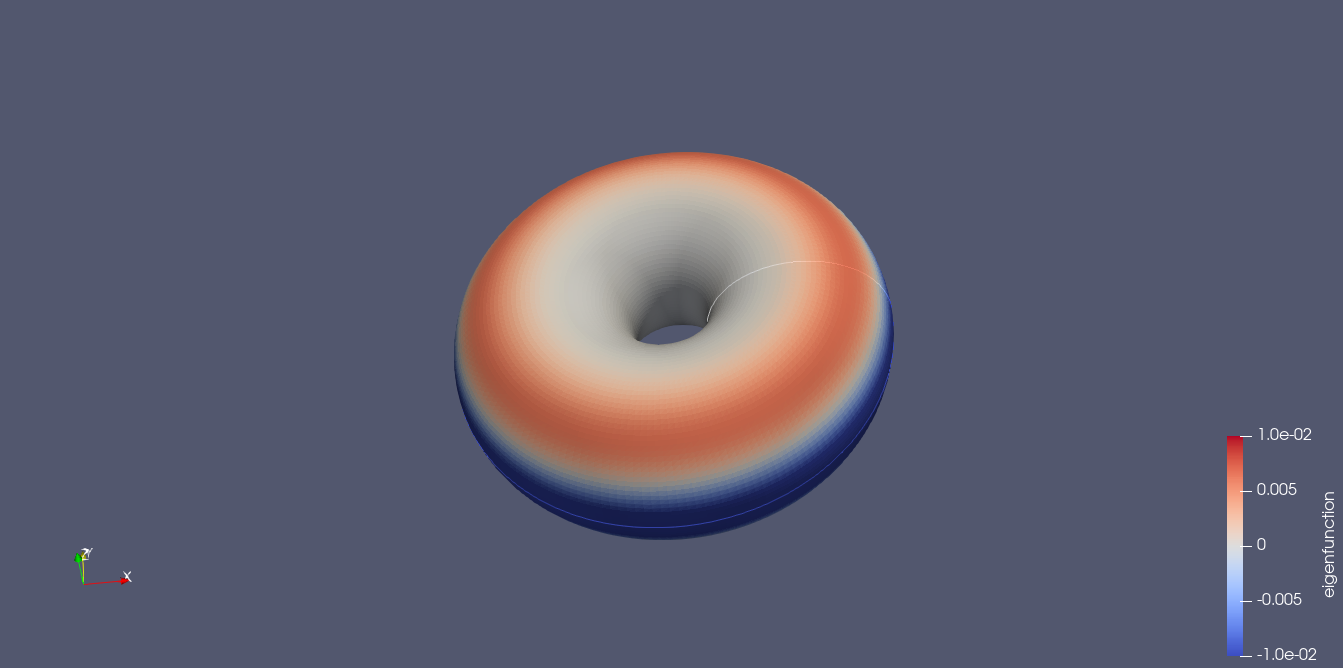}\hspace{.01\textwidth}
\includegraphics[width=.30\textwidth]{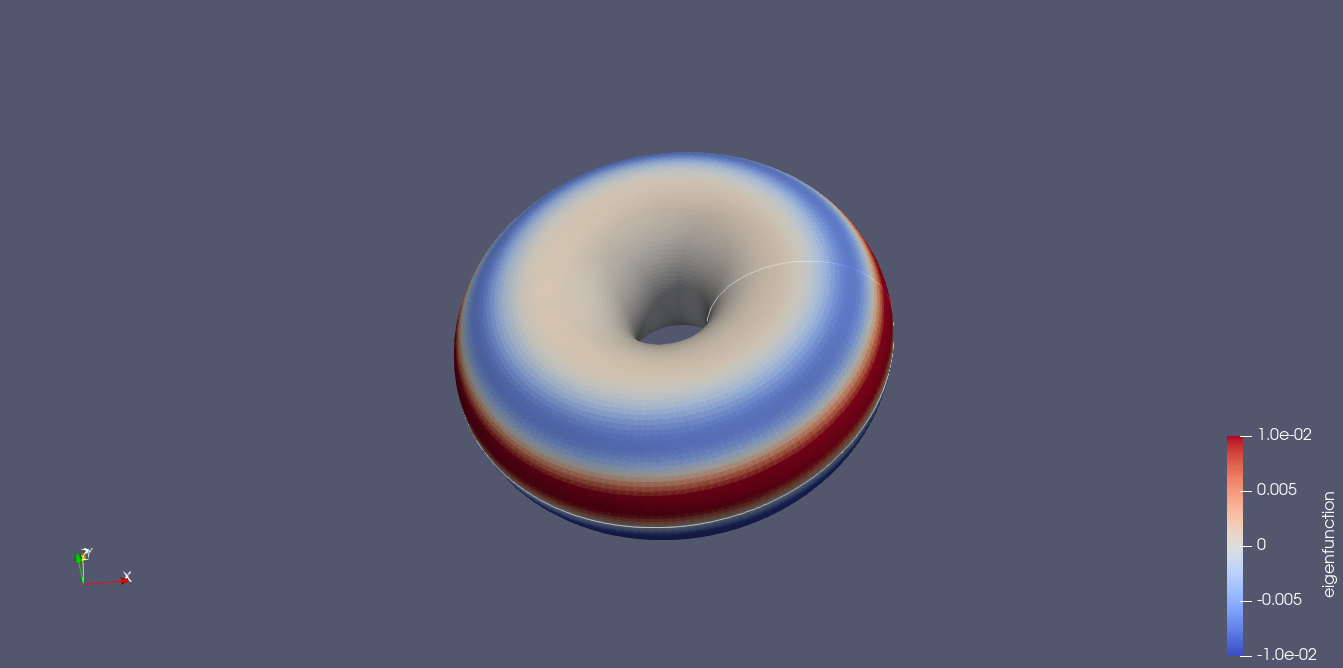}\hspace{.01\textwidth}
\includegraphics[width=.30\textwidth]{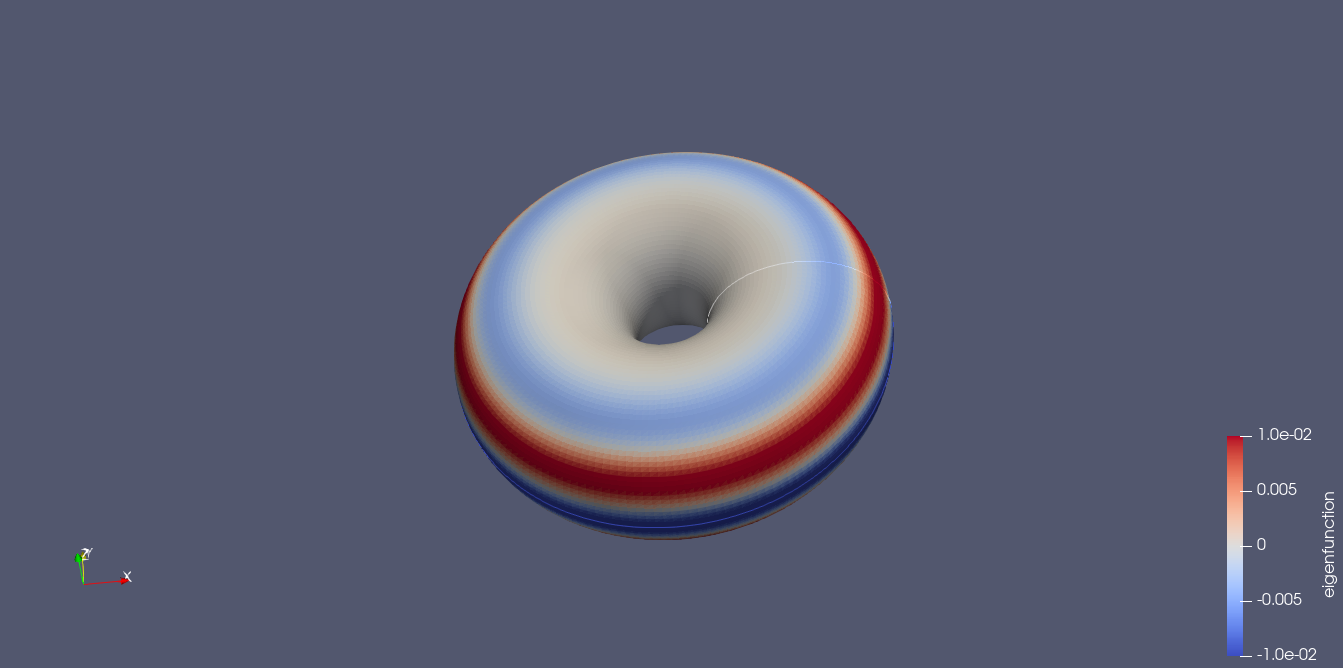} \\ \vspace{.01\textwidth}
\includegraphics[width=.30\textwidth]{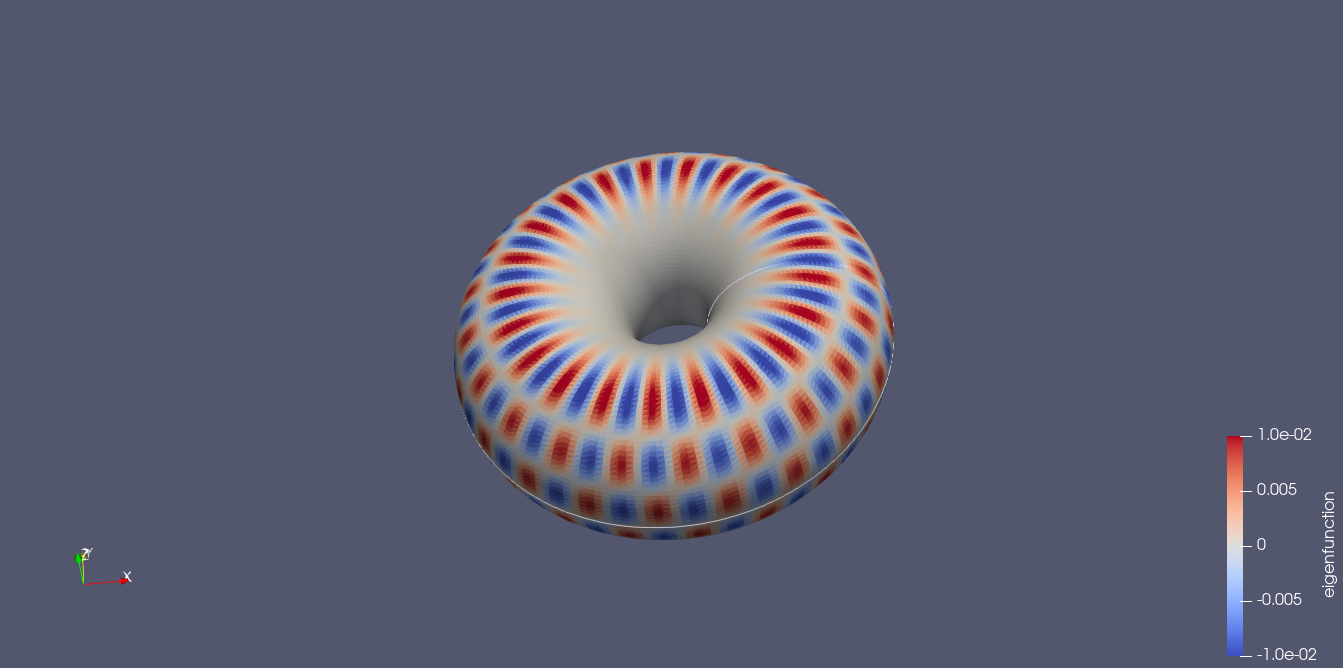}\hspace{.01\textwidth}
\includegraphics[width=.30\textwidth]{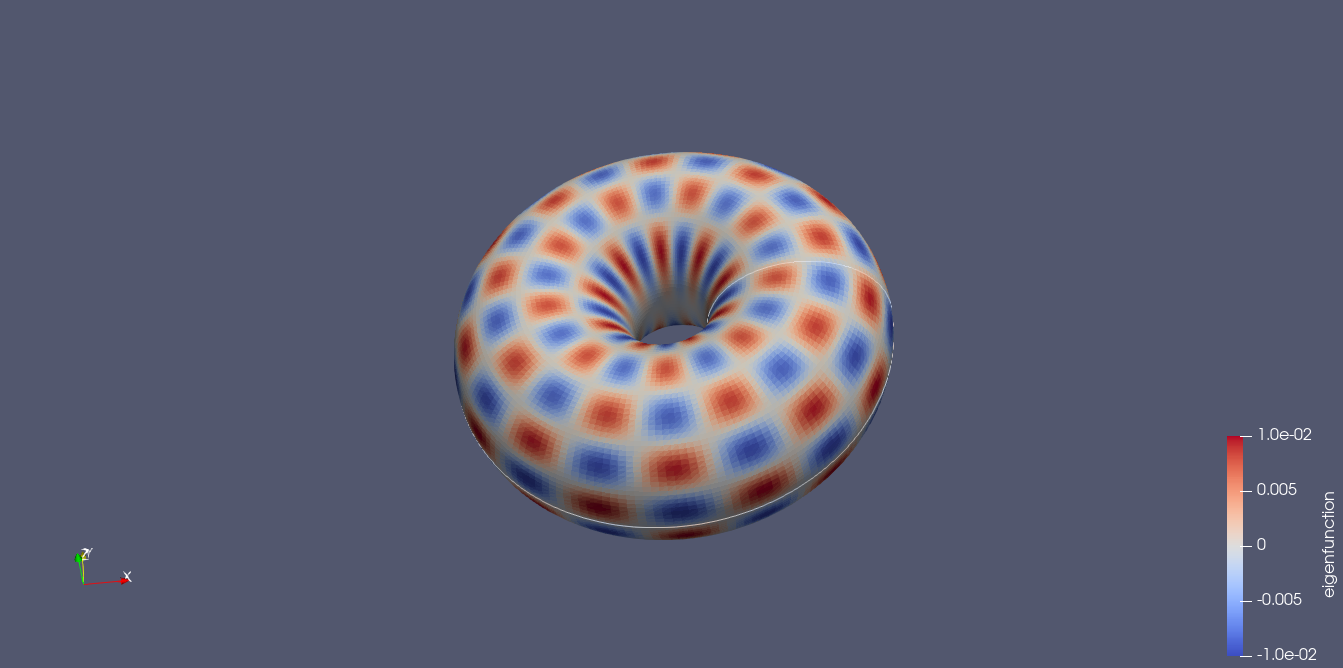}
\caption{Top: The eigenfunctions corresponding to exceptional $100$th, $305$th and $402$nd eigenvalues on the Clifford torus. They are all invariant under the rotation with respect to the $z-$axis. Bottom: The $401$st and $403$rd eigenfunctions. They are not invariant under the rotation with respect to the $z-$axis.}
\label{fig:eigenfunction2}
\end{figure}

\begin{figure}[htp]
\centering
\includegraphics[width=.39\textwidth]{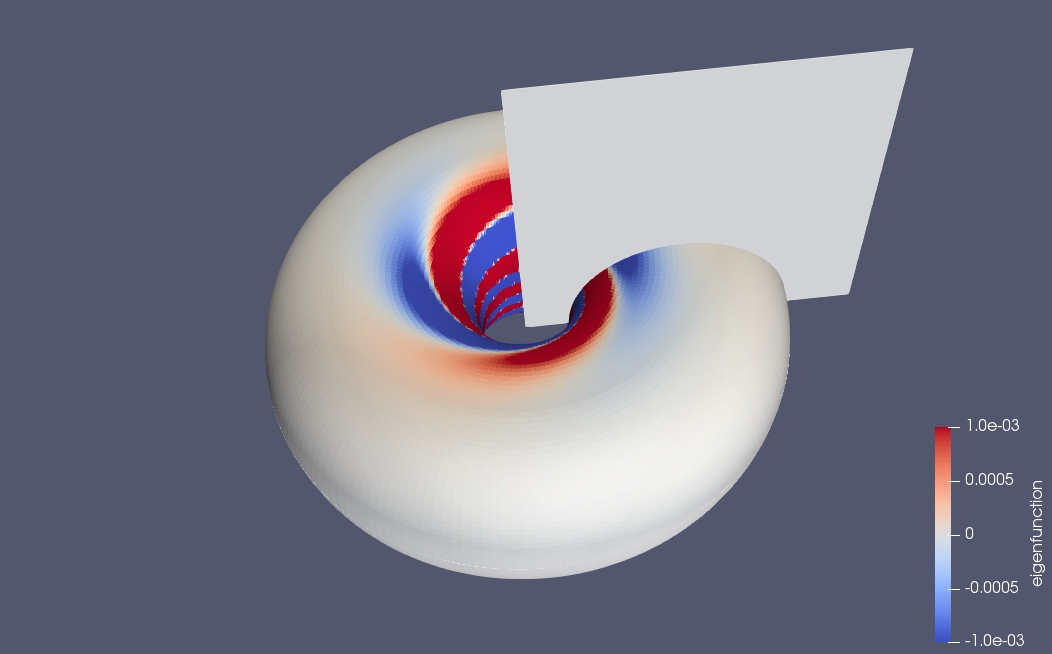}\hspace{.01\textwidth}
\includegraphics[width=.39\textwidth]{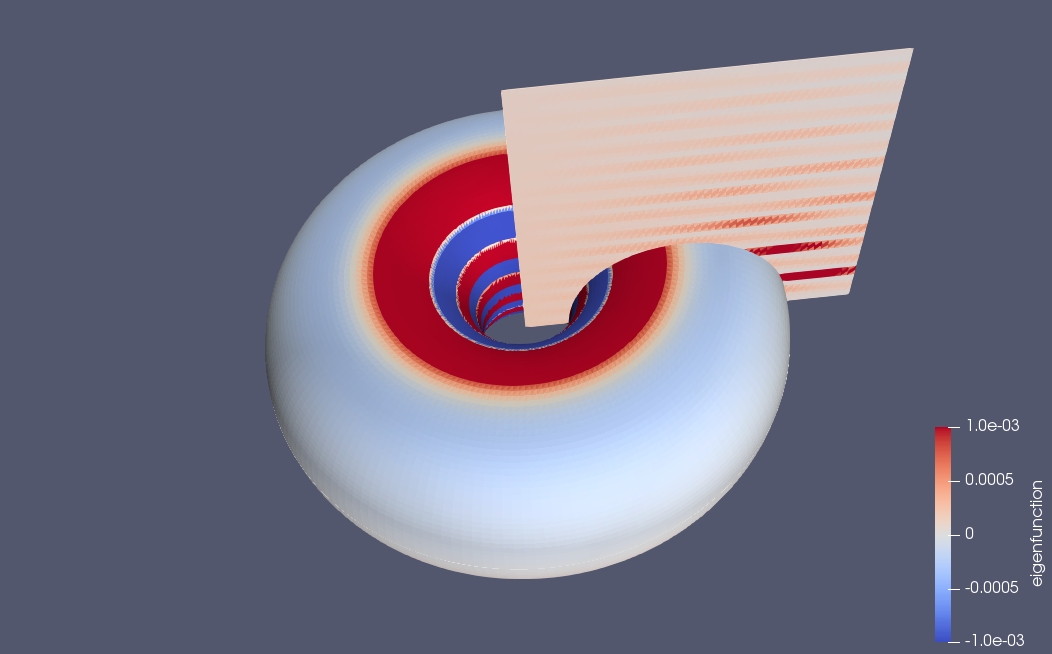}
\caption{The $-38$th eigenfunction (left) and the $-39$th eigenfunction (right) on the Clifford torus. The $-39$th eigenvalue is an exceptional one.}
\label{fig:eigenfunction3}
\end{figure}

It is insightful to compare eigenfunctions on the torus and those on the ellipsoid. Figure \ref{fig:eigenfunction_ell} shows the 56th and the 123rd eigenfunctions. They seem to be supported more locally than eigenfunctions on the torus without exhibiting symmetries.

\begin{figure}[htp]
\centering
\includegraphics[width=.39\textwidth]{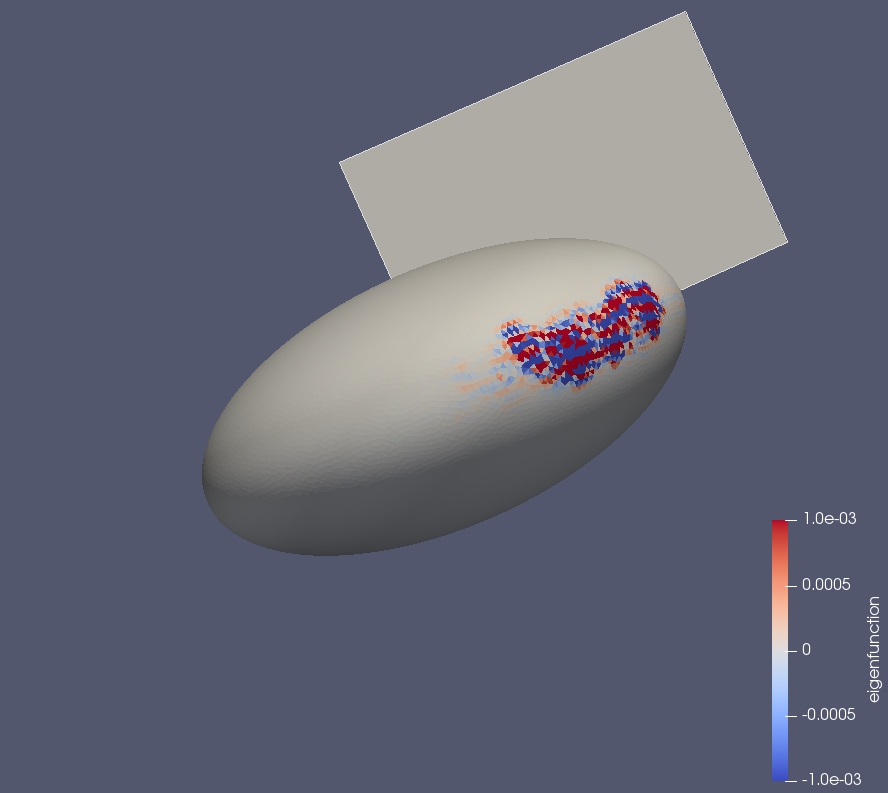}\hspace{.01\textwidth}
\includegraphics[width=.39\textwidth]{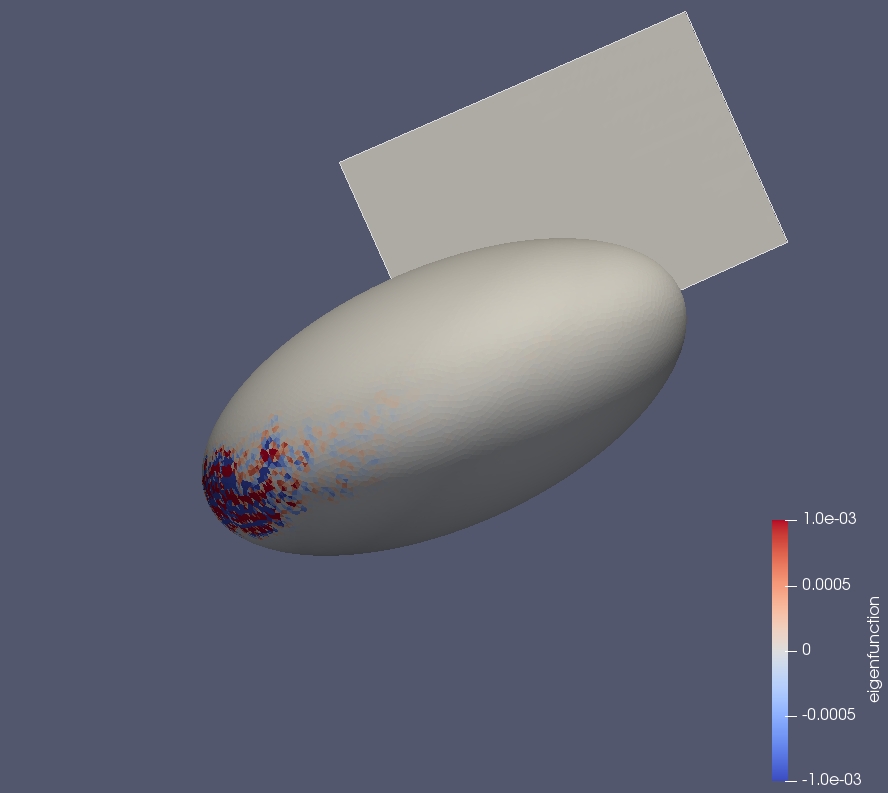}
\caption{The 56th eigenfunction (left) and the 123rd one (right) on the ellipsoid. }
\label{fig:eigenfunction_ell}
\end{figure}

Some words on numerical computation are in order. We use about $2 \times 10^6$ meshes on the surfaces (the torus and the ellipsoid) and regions of integration ($X$ and $Y$). Meshes are generated automatically using AMR (Adaptive Mesh Refinement). The number  $\Ge$ appearing in the sets $X$ and $Y$ in \eqnref{setX} and \eqnref{setY} is approximately the side length of the meshes (this is why it is about $1/\sqrt{2,000,000}$). We employ the Gaussian quadrature rule for approximations of the definite integrals in the NP operator and the single layer potential. The NP eigenvalues and eigenfunctions are obtained by numerically computing the eigenvalues and eigenfunctions of the matrix obtained by discretizing the NP operator. These numerical calculations are performed using Freefem++4.4.2 \cite{Hecht} and LAPACK, version 3.8.0 \cite{Lapack}. The computed eigenvalues have small imaginary parts which appear because of the discretization error. We take only the real parts as eigenvalues.

\section*{Discussion}

We studied in a quantitative manner the decay of the plasmon $\Scal_{\p\GO}[\Gvf_j](z)$ in three dimensions. We showed that it decays at the rate of $j^{-1/2}$ almost surely for smooth surface $\p\GO$ of arbitrary shape. The convergence rate becomes $j^{-\infty}$ if $\GO$ is strictly convex. We showed as a consequence that CALR does not occur on strictly convex surfaces. We also found out through numerical computations that there are four exceptional eigenvalues among the first 450 positive eigenfunctions on the Clifford torus for which the plasmon have large values. The common feature of eigenfunctions corresponding to these exceptional eigenvalues is that they are invariant under the rotation with respect to the axis of symmetry. It would be quite important and challenging to investigate rigorously these computational findings on the torus. It may be related to possibility of occurrence of CALR on tori.



\begin{thebibliography}{11}

\bibitem{Ahner} J. F. Ahner, On the eigenvalues of the electrostatic integral operator. II. J. Math. Anal. Appl.  181  (1994). 328--334.
\bibitem{ACKLM-13} H. Ammari, G. Ciraolo, H. Kang, H. Lee and G. W. Milton, Spectral theory of a Neumann-Poincar\'e-type operator and analysis of cloaking due to anoumalous localized resonance, Arch. Ration. Mech. An. 208 (2013), 667--692.

\bibitem{ACKLM14} H. Ammari, G. Ciraolo, H. Kang, H. Lee and G. W. Milton, Spectral theory of a Neumann-Poincar\'e-type operator and analysis of anomalous localized resonance II, Contemporary Math. 615, 1--14, 2014.

\bibitem{AmKa07Book2} {H. Ammari and H. Kang}, {\sl Polarization and moment tensors},
{Applied Mathematical Sciences}, 162, Springer, New York. 2007.

\bibitem{AK} K. Ando and H. Kang,
\newblock Analysis of plasmon resonance on smooth domains using spectral properties of the Neumann-Poincar\'e operator,
\newblock J. Math. Anal. Appl., 435 (1) (2016), 162--178.

\bibitem{AJKKM} K. Ando, Y. G. Ji, H. Kang, D. Kawagoe and Y. Miyanishi, Spectral structure of the Neumann--Poincar\'e operator on tori, Ann. I. H. Poincare-AN 36 (2019), 1817--1828.

\bibitem{Lapack}
E. Anderson, Z. Bai, C. Bischof, S. Blackford, J. Demmel, J. Dongarra, J. Du Croz, A. Greenbaum,  S. Hammarling, A. McKenney, and D. Sorensen, {\sl LAPACK Users' Guide (Third). Philadelphia, PA: Society for Industrial and Applied Mathematics} (1999).

\bibitem{CKKL} D. Chung, H. Kang, K. Kim and H. Lee, Cloaking due to anomalous localized resonance in plasmonic structures of confocal ellipses, SIAM J. Appl. Math. 74 (2014), 1691--1707.

\bibitem{Agrano-book}
\newblock Yu. V. Egorov and M. A. Shubin (Eds.),
\newblock{\sl Encyclopaedia of Mathematical Sciences Vol. 63, Partial Differential Equations VI, Elliptic and Parabolic Operators}, Springe-Verlag, Berlin, 1994.

\bibitem {Fo95} {G. B. Folland},
{\sl Introduction to partial differential equations}, 2nd Ed., Princeton Univ. Press, Princeton, 1995.

\bibitem{Hecht} F. Hecht, New development in FreeFem++, J. Numer. Math., (3-4) 20,  (2012), 251--265.

\bibitem{JK} Y. G. Ji and H. Kang, A concavity condition for existence of a negative value  in Neumann-Poincar\'e spectrum in three dimensions, Proc. Amer. Math. Soc 147 (2019), 3431-3438.

\bibitem{KKLSY-JLMS-16} H. Kang, K. Kim, H. Lee, J. Shin and S. Yu, Spectral properties of the Neumann--Poincar\'e operator and uniformity of estimates for the conductivity equation with complex coefficients, J. London Math. Soc. (2) 93 (2016), 519--546.

\bibitem{KPS-ARMA-07} {D. Khavinson, M. Putinar and H. S. Shapiro},
    {Poincar\'e's variational problem in potential theory},
  {Arch. Ration. Mech. An.} 185 (2007), 143--184.

\bibitem{KLSW} R. V. Kohn, J. Lu, B. Schweizer and M. I. Weinstein, A variational perspective
on cloaking by anomalous localized resonance, Comm. Math. Phys. 328 (2014), 1--27.

\bibitem{MN}
F. C. Marques and A. Neves,
\newblock Min-Max theory and the Willmore conjecture,
\newblock Anal. Math. 179 (2014), 683--782.

\bibitem{MM} R. C. McPhedran and G. W. Milton, A review of anomalous resonance, its associated cloaking, and superlensing, arXiv:1910.13808v3.

\bibitem{MN_06} G. W. Milton and N. A. P. Nicorovici, On the cloaking
effects associated with anomalous localized resonance, Proc. R.
Soc. A 462 (2006), 3027--3059.

\bibitem{Miyanishi:Weyl}Y.~{Miyanishi},
\newblock Weyl's law for the eigenvalues of the Neumann--Poincar\'e operators in three dimensions: Willmore energy and surface geometry,
\newblock arXiv:1806.03657. 

\bibitem{MR}Y. Miyanishi and G. Rozenblum,
\newblock Eigenvalues of the Neumann-Poincare operator in dimension 3: Weyl's law and geometry,
\newblock St. Petersburg Math. J., 31(2) (2019), 248--268. 

\bibitem{Verch-JFA-84} {G. C. Verchota}, {Layer potentials and boundary value problems for
Laplace's equation in Lipschitz domains}, J. Funct. Anal. 59 (1984), 572--611.

\bibitem{Zelditch}{S. Zelditch}, {\sl Encyclopedia of mathematical physics} Vol. 1, 2, 3, 4, 5.
Edited by Jean-Pierre Fran\c{c}oise, Gregory L. Naber and Tsou Sheung Tsun. Academic Press/Elsevier Science, Oxford, 2006. Vol. 1: l+679 pp.; Vol. 2: l+729 pp.; Vol 3: l+645 pp.; Vol. 4: l+673 pp.; Vol. 5: l+549 pp. ISBN: 978-0-1251-2660-1; 0-12-512660-3. 

\end{thebibliography}
\end{document}